\documentclass[11pt]{article}

\usepackage[margin=1in]{geometry}
\usepackage{graphicx}
\usepackage{amsmath,amssymb,amsfonts,amsthm}
\usepackage{mathtools}
\usepackage{mathrsfs}
\usepackage{xcolor}
\usepackage{booktabs}
\usepackage{csvsimple}
\usepackage{siunitx}
\usepackage[natbib=true,style=numeric,sorting=none]{biblatex}
\usepackage[hidelinks]{hyperref}

\newtheorem{theorem}{Theorem}
\newtheorem{proposition}[theorem]{Proposition}
\newtheorem{lemma}{Lemma}
\newtheorem{corollary}{Corollary}
\theoremstyle{definition}

\theoremstyle{remark}

\DeclareMathOperator{\Tr}{tr}
\newcommand{\mat}[1]{{\boldsymbol{#1}}}
\newcommand{\vect}[1]{{\mathbf #1}}

\newcommand{\inprod}[2]{#1^T #2}
\DeclareMathOperator{\expct}{\mathbb E}
\newcommand{\expect}[1]{\expct \left[#1\right]}
\newcommand{\EQ}[1]{(\ref{#1})}
\newcommand{\bx}{{\vect x}}
\newcommand{\norm}[1]{\lVert#1\rVert}

\newcommand{\highlight}[1]{#1}
\newcommand{\va}{\vect{a}}

\newcommand{\mA}{\mat{A}}
\newcommand{\mG}{\mat{G}}
\newcommand{\mP}{\mat{P}}
\newcommand{\mQ}{\mat{Q}}

\newcommand{\I}{\mat{I}}
\newcommand{\GG}{\boldsymbol{{\Phi}}}

\title{Volume Sampling and Spectral Equalization in Randomized Alternating Projections}
\author{Alireza Entezari\thanks{Corresponding author: \texttt{entezari@ufl.edu}} \and Arunava Banerjee \and Leila Kalantari\\[0.5em]\normalsize Department of CISE, University of Florida, Gainesville, FL, USA}
\date{}

\begin{document}
\maketitle

\begin{abstract}
We study a family of randomized alternating projection methods that alternate among subspaces spanned by subsets of \(n\) vectors from a prescribed set. We derive an explicit characterization of performance bounds when these subspaces are sampled with probabilities proportional to the volumes subtended by the vectors in their corresponding subsets. For each \(n\), the bound is obtained through an explicit nonlinear transformation of the spectrum of an associated matrix. Our analysis reveals an explicit spectral equalization mechanism that drives the spectrum toward improved conditioning as $n$ increases, establishing an unexpected connection between volume sampling and the Cayley-Hamilton theorem. Furthermore, we introduce a computationally efficient uniform sampling scheme that achieves comparable theoretical guarantees through a specific relaxation strategy. Empirical results show that this relaxation method achieves convergence rates comparable to, and often better than, volume sampling, while avoiding the infeasibility of volume sampling at large scale. Besides the randomized Kaczmarz, these results also directly apply to randomized Gauss-Seidel and coordinate descent methods.
\end{abstract}


\section{Introduction}\label{sec1}
Iterative methods are a cornerstone of computational mathematics, with broad applications to large-scale problems in scientific computing, optimization, machine learning, and imaging. Over the past two decades, randomized techniques have significantly advanced both the theory and practice of iterative methods. In particular, randomized variants of several classical algorithms, including Kaczmarz~\cite{strohmer2009randomized}, Gauss-Seidel~\cite{leventhal2010randomized, tu2017breaking}, and coordinate descent~\cite{nesterov2012efficiency}, have been developed to address large-scale and streaming problems, where low memory requirements and simple iterations are essential.

Convergence analysis of iterative methods is a classical topic~\cite{varga1962iterative, saad2003iterative, hackbusch2016iterative}. For stationary methods such as Jacobi, Gauss--Seidel, and successive over-relaxation (SOR), convergence is governed by the spectral radius of the corresponding iteration matrix. A fundamental challenge is to relate this spectral radius to readily computable properties of the underlying problem, such as the spectrum or condition number of the system matrix~\cite{ortega1979extensions}. While such relationships are available in important special cases, they typically rely on additional structural assumptions, such as symmetry, positive definiteness, or consistent ordering~\cite{berman1994nonnegative, hackbusch2016iterative}.

For projection methods such as the Kaczmarz algorithm, the challenge takes a different form. Convergence is governed by the geometry of successive projections onto subspaces~\cite{deutsch1992method}: the convergence rate depends not only on the angles between subspaces but also on the order in which they, or equivalently the rows of the matrix, are visited~\cite{deutsch1995angle, galantai2005rate}. Existing convergence bounds are therefore expressed in terms of geometric quantities, such as sequences of principal angles encountered during a projection cycle~\cite{halperin1962product, deutsch1997rate, ssw1977, kayalar1988error}. Although these bounds provide valuable geometric insight, they are often difficult to evaluate in practice~\cite{galantai2005rate} and do not directly relate convergence to the spectral properties of the underlying system.

Randomization offers a way to circumvent this difficulty by enabling convergence rates to be expressed in terms of spectral properties of the underlying problem. In their pioneering work, Strohmer and Vershynin~\cite{strohmer2009randomized} showed that the Kaczmarz algorithm for solving \(\mA \bx = \vect{b}\), when randomized by selecting individual rows with probabilities proportional to the squares of their Euclidean norms, converges exponentially in expectation. This result expresses the convergence rate in terms of the smallest singular value of \(\mA\), or equivalently its scaled condition number. This randomization was subsequently extended to coordinate descent, yielding a randomized form of the Gauss--Seidel algorithm~\cite{leventhal2010randomized}. Since then, randomized iterative methods have attracted considerable attention, leading to numerous variants, acceleration techniques~\cite{liu2016accelerated, tu2017breaking}, and generalizations~\cite{gower2015randomized}.

We show that the randomized Kaczmarz method of Strohmer and Vershynin is the first member of a hierarchy of randomized projection methods in which additional work per iteration leads to explicitly characterized improvements in conditioning. Specifically, we generalize randomized Kaczmarz by updating the iterate at each step with the orthogonal projector associated with a subset of \(n\) rows of \(\mA\), selected with probability proportional to the squared volume subtended by that subset. Our principal theoretical contribution is an exact characterization of the expected projector induced by volume sampling. We show that volume sampling induces a spectral transformation that progressively equalizes the singular-value spectrum of \(\mA\), and that this transformation is described by a recursive form of the Cayley--Hamilton theorem, yielding the first exact characterization of its effect on conditioning. While volume sampling has been studied for its theoretical advantages in a variety of algorithmic settings~\cite{frieze2004fast, rudelson2007sampling, deshpande2006adaptive, derezinski2018reverse}, and specifically for randomized coordinate descent~\cite{rodomanov2020randomized}, our analysis reveals its precise effect on conditioning.

From a practical perspective, the combinatorial nature of subsets limits the applicability of volume sampling to large-scale problems. Although Markov chain sampling techniques have been developed~\cite{deshpande2006adaptive, deshpande2010efficient} to mitigate this combinatorial difficulty, their computational overhead remains a major obstacle in high-dimensional and streaming settings. We therefore introduce a simple uniform sampling scheme that, through a relaxation, approximates the effect of volume sampling while retaining comparable theoretical guarantees. The tradeoff lies in the theoretical convergence bound, while in our experiments the relaxed method consistently outperforms volume sampling.

The remainder of the paper is organized as follows. Section~\ref{sec:preliminaries} introduces the notation used throughout the paper and briefly reviews the randomized Kaczmarz method. Section~\ref{sec:results} presents the main results, including the exact characterization of the expected projector under volume sampling, the resulting spectral equalization, and the relaxed uniform sampling scheme. Section~\ref{sec:proofs} contains the proofs of the main results. Section~\ref{sec:experiments} presents numerical experiments. Finally, Section~\ref{sec:discussion} places our results in the context of related work, discusses their implications, and outlines directions for future research.

\section{Preliminaries}\label{sec:preliminaries}

Vectors are denoted by bold lowercase letters and matrices by bold uppercase letters. The transpose of a matrix \(\mA\) is denoted by \(\mA^T\). For a symmetric matrix \(\mA\), we write \(\lambda_{\min}(\mA):=\lambda_1(\mA)\le\cdots\le\lambda_N(\mA)\) for its eigenvalues. The identity matrix is denoted by \(\I\), and \(\Tr(\mA)\) denotes the trace of \(\mA\).

For a consistent system of equations \(\mA\bx=\vect{b}\), where \(\mA\in\mathbb{R}^{M\times N}\) with \(M\ge N\), the Kaczmarz algorithm generates iterates by successive orthogonal projections onto the affine hyperplanes defined by the rows of \(\mA\) and the corresponding entries of \(\vect{b}\). Let \(\mathcal{A}:=\{\va_1^T,\dots,\va_M^T\}\) denote the collection of row vectors of \(\mA\). At each step, the randomized Kaczmarz method selects a row \(\va^T\in\mathcal{A}\) according to a prescribed probability distribution and updates the iterate by
\begin{equation}
\bx_{k+1}
:=
\bx_k
+
\frac{b_{\va}-\va^T\bx_k}{\va^T\va}\,\va,
\label{eq:kacz}
\end{equation}
where \(b_{\va}\) denotes the entry of \(\vect{b}\) corresponding to the selected row. The iterates converge to a point \(\bx_\star\) in the intersection of the hyperplanes.

When rows are selected independently with probabilities \(p_\va=\norm{\va}^2/\norm{\mA}_F^2\), where \(\norm{\mA}_F^2=\sum_{\va^T\in\mathcal A}\norm{\va}^2=\Tr(\mA^T\mA)\), the convergence rate admits the bound
\begin{equation}\label{eq:SV}
    \expect{\norm{\bx_{k+1}-\bx_\star}^2\mid\bx_k}
    \le
    (1-\mu_1)\norm{\bx_k-\bx_\star}^2,
    \qquad
    \mu_1
    :=
    \frac{\lambda_{\min}(\mA^T\mA)}
         {\Tr(\mA^T\mA)}.
\end{equation}
Since the rows are sampled independently at each iteration, this one-step estimate immediately yields the global convergence bound
\[
\expect{\norm{\bx_k-\bx_\star}^2}
\le
(1-\mu_1)^k
\norm{\bx_0-\bx_\star}^2.
\]

For an integer \(n\ge1\), let
\[
\mathcal{A}^{(n)}
:=
\left\{
[\va_{i_1}\ \cdots\ \va_{i_n}]^T
\in\mathbb{R}^{n\times N}
:
1\le i_1<\cdots<i_n\le M
\right\}
\]
denote the collection of the \(\binom{M}{n}\) submatrices of \(\mA\) with \(n\) rows. By a slight abuse of notation, we identify each element of \(\mathcal{A}^{(n)}\) with the corresponding matrix \(\mA_n\), and denote by \(\vect{b}_{\mA_n}\) the corresponding subvector of \(\vect{b}\). Projection onto the affine subspace defined by \((\mA_n,\vect{b}_{\mA_n})\) yields the iteration
\begin{equation}\label{eq:multi_row}
    \bx_{k+1}
    :=
    \bx_k
    +
    \mA_n^T(\mA_n\mA_n^T)^\dagger
    \left(
        \vect{b}_{\mA_n}
        -
        \mA_n\bx_k
    \right),
\end{equation}
where \(\dagger\) denotes the Moore--Penrose pseudoinverse. The randomized methods developed in this paper are obtained by selecting the matrices \(\mA_n\in\mathcal{A}^{(n)}\) according to carefully designed probability distributions.

\section{Results}\label{sec:results}

\subsection{Volume Sampling}
Generalizing the notion of squared Euclidean norm of row vectors, we show when the subsets $\mA_n \in \mathcal{A}^{(n)}$ are sampled according to the square of volume their corresponding set of rows subtend, $\det(\mA_n\mA_n^T)$, the rates of convergence for iterations in~\EQ{eq:multi_row} are explicitly determined from the spectrum of $\mA$ and $n$. Our main result shows explicitly a { nonlinear transformation} that occurs to the spectrum of $\mA$ for each $n>1$.
More specifically, let $\lambda^{(1)}_1 \le \dots \le \lambda^{(1)}_N$ denote the eigenvalues of $\mA^T \mA$ with $\mu_1 = \lambda^{(1)}_1 / \sum_i \lambda^{(1)}_i$ providing the rate in~\EQ{eq:SV} for iterations with individual rows. For $n=2$ the transformed eigenvalues are given by the quadratic polynomial $\lambda_i^{(2)} := \lambda_i^{(1)}(\sum_{j=1}^N\lambda_j^{(1)} - \lambda_i^{(1)})$ which preserves the ordering of eigenvalues and yields the convergence rate $\mu_2 := 2\lambda^{(2)}_1 / \sum_i \lambda^{(2)}_i$. The following theorem generalizes this construction to arbitrary subset sizes.

\begin{theorem}\label{thm:volume_sampling}
For a fixed \(n>1\), consider the randomized alternating projections in~\EQ{eq:multi_row}, where each \(\mA_n \in \mathcal{A}^{(n)}\) is sampled with squared volume probability
\[
p_{\mA_n}
:=
\frac{v_{\mA_n}}
{\sum_{\mA_n' \in \mathcal{A}^{(n)}} v_{\mA_n'}}\qquad {\rm with} \quad v_{\mA_n} := \det(\mA_n \mA_n^T).
\]
Then the conditional expected error satisfies
\begin{equation*}
    \expect{\norm{\bx_{k+1} - \bx_\star}^2 \mid \bx_k}
    \le
    (1-\mu_n)\norm{\bx_k-\bx_\star}^2.
\end{equation*}
The rate parameter is determined, recursively, from the eigenvalues $\{\lambda^{(1)}_i\}_{i=1}^N$ of $\mA^T\mA$ by
\begin{equation}\label{eq:evolution}
    \mu_n
    :=
    n \frac{\lambda_1^{(n)}}{\sum_{j=1}^N \lambda_j^{(n)}},
    \qquad
    \lambda_i^{(n)}
    :=
    \lambda_i^{(1)}
    \left(
        \frac{\sum_{j=1}^N \lambda_j^{(n-1)}}{n-1}
        -
        \lambda_i^{(n-1)}
    \right).
\end{equation}
\end{theorem}

The nonlinear transformation of eigenvalues of \(\mA^T\mA\) described in~\EQ{eq:evolution} drives the spectrum toward equality as \(n\) increases, with complete equalization at \(n=N\). This evolution is governed by a recursive form of the Cayley–Hamilton theorem, involving the lowest-degree \(n\) terms of the characteristic polynomial of \(\mA^T\mA\), as used in the Faddeev–LeVerrier algorithm for calculating coefficients of the characteristic polynomial.

\subsection{Spectral Equalization}
\begin{figure}[t]
    \centering
    \includegraphics[width=\textwidth]{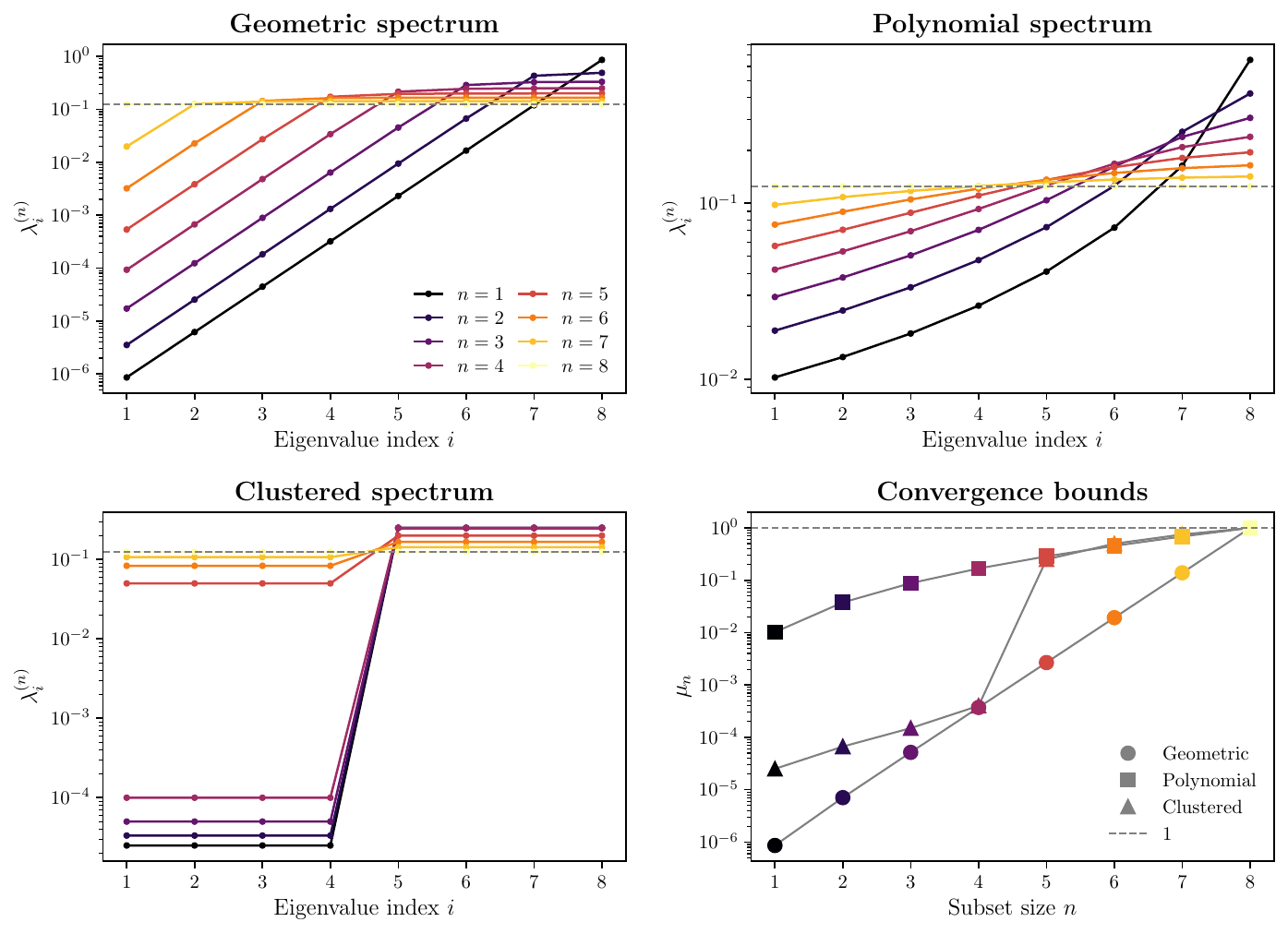}
    \caption{
    Evolution of the normalized transformed spectra under the volume-sampling transformation.
    For geometric (top left), polynomial (top right), and clustered (bottom left) initial spectra of \(\mA^T\mA\), each curve shows
    \(\lambda_i^{(n)} / \sum_{j=1}^N \lambda_j^{(n)}\)
    for a subset size \(n=1,\ldots,N\), with \(N=8\).
    The normalized spectra are plotted on a logarithmic vertical scale.
    The dashed horizontal line marks the uniform value \(1/N\), illustrating the spectral equalization mechanism as \(n\) increases. The bottom right plot shows the improvements on the bound $\mu_n$ on the geometric rate of convergence given by Theorem~\ref{thm:volume_sampling}.
    }
    \label{fig:spectral_equalization}
\end{figure}

The purpose of this section is to illustrate the nonlinear spectral transformation underlying Theorem~\ref{thm:volume_sampling}. We examine how increasing the subset size \(n\) modifies the spectrum and the resulting convergence parameter \(\mu_n\). These experiments isolate the spectral mechanism responsible for the improved convergence rates established by the theory.

In all cases we take \(N=8\). The geometric spectrum consists of logarithmically spaced eigenvalues between \(10^{-6}\) and \(1\), corresponding to condition number \(10^6\). The polynomial spectrum is chosen from the values \(\{i^{-2}\}_{i=1}^N\), sorted in increasing order. The clustered spectrum consists of \(N/2\) eigenvalues equal to \(10^{-4}\) and \(N/2\) eigenvalues equal to \(1\). These choices represent, respectively, a severely ill-conditioned spectrum with smooth decay, a moderately ill-conditioned polynomially decaying spectrum, and a spectrum with a sharp separation between small and large eigenvalue clusters.

For each spectrum, we plot the normalized quantities $\lambda_i^{(n)}/\sum_{j=1}^N \lambda_j^{(n)}$, so that the figure displays the distribution of spectral mass rather than its scale. The dashed horizontal line marks the uniform value \(1/N\), which corresponds to complete spectral equalization. We also report the convergence parameter $\mu_n := {\lambda_1^{(n)}}/({\sum_{j=1}^N \lambda_j^{(n)}})$,
which governs the contraction factor in the bound of Theorem~\ref{thm:volume_sampling}.

Figure~\ref{fig:spectral_equalization} illustrates how the transformation pushes the normalized spectra toward the uniform distribution as \(n\) increases. Correspondingly, \(\mu_n\) increases toward \(1\), illustrating how the volume-sampling transformation improves the conditioning that controls the convergence rate.

\subsection{Relaxed Uniform Sampling}
The spectral equalization mechanism described above motivates the use of volume sampling. However, exact volume sampling becomes computationally impractical in high-dimensional settings because of the combinatorial growth in the number of subsets. We therefore introduce a relaxation strategy for the iterations in~\EQ{eq:multi_row} that enables uniform sampling while retaining a comparable convergence guarantee and avoiding the computational burden associated with volume sampling.

For the collection \(\mathcal{A}^{(n)}\), we denote the mean and maximum squared volumes by
\[
\bar{v}_n
:=
\frac{1}{|\mathcal{A}^{(n)}|}
\sum_{\mA_n \in \mathcal{A}^{(n)}} v_{\mA_n},
\qquad
v_n^{\max}
:=
\max_{\mA_n \in \mathcal{A}^{(n)}} v_{\mA_n}.
\]

\begin{theorem}\label{thm:uniform_sampling}
Consider the relaxed iterations
\begin{equation}\label{eq:relaxed}
    \bx_{k+1}
    :=
    \bx_k
    +
    \omega_{\mA_n}\mA_n^T(\mA_n\mA_n^T)^\dagger
    \left(\vect{b}_{\mA_n} - \mA_n \bx_k\right),
    \qquad
    \omega_{\mA_n}
    :=
    1 + \sqrt{1 - \frac{v_{\mA_n}}{v_n^{\max}}},
\end{equation}
where \(\mA_n \in \mathcal{A}^{(n)}\) is sampled uniformly at each iteration. Then
\[
        \expect{\norm{\bx_{k+1} - \bx_\star}^2 \mid \bx_k}
        \le
        \left(1-\widetilde{\mu}_n\right)
        \norm{\bx_k-\bx_\star}^2 \qquad {\rm with} \quad \widetilde{\mu}_n := \frac{\bar{v}_n}{v_n^{\max}}\mu_n.
\]
\end{theorem}
We note that the squared volume is obtained as a by-product of the projection step, namely the solution of an \(n \times n\) system, and therefore does not incur additional computation. Also the maximum squared volume \(v_n^{\max}\), required for the relaxation in~\EQ{eq:relaxed}, need not be known in advance. Instead, the algorithm can maintain a running maximum, updating it before each relaxed step. This online normalization preserves stability by ensuring \(\omega_{\mA_n} \le 2\), while avoiding the need to pre-compute \(v_n^{\max}\) over all subsets. 

While experimental validation shows that the iterations in~\EQ{eq:relaxed} consistently  outperform the volume-sampling iterations in~\EQ{eq:multi_row}, the tradeoff is a weaker bound,  with $\mu_n$ reduced by the factor $\frac{\bar{v}_n}{v_n^{\max}}$.

\section{Proofs of the Main Results}\label{sec:proofs}
The analysis of iterations is more conveniently expressed in the centered form around the solution $\bx_\star$. Let $\mP_{\mA_n}$ denote the orthogonal projector onto the subspace spanned by the subset $\mA_n \in \mathcal{A}^{(n)}$, chosen at step $k$. Since $\vect{b}_{\mA_n} = \mA_n \bx_\star$, the centered form of~\EQ{eq:multi_row} is:
\begin{equation}
    \bx_{k+1} - \bx_\star = (\I - \mP_{\mA_n})(\bx_k - \bx_\star),\qquad {\rm where} \quad \mP_{\mA_n} := \mA_n^T(\mA_n\mA_n^T)^\dagger\mA_n.
\end{equation}
For each choice of $\mA_n \in \mathcal{A}^{(n)}$ for the $k^{\rm th}$-step, the progress can be measured by $\norm{\bx_{k+1} - \bx_\star}^2 = (\bx_{k} - \bx_\star)^T(\I - \mP_{\mA_n}) (\bx_{k} - \bx_\star)$. Then the conditional expected progress of the algorithm, in each step, for a sampling distribution $p$ for choosing $\mA_n \in \mathcal{A}^{(n)}$, is characterized by \begin{equation}
    \mathbb{E}_{\mA_n \sim p}[{\norm{\bx_{k+1} - \bx_\star}^2 \mid \bx_k}]
=
(\bx_k - \bx_\star)^T
(\I - \overline{\mP}_n)
(\bx_k - \bx_\star).
\end{equation}
Here the \emph{expected projector} \(\overline{\mP}_n\) is the average of the projectors, corresponding to $n$-rows, under the sampling distribution $p$ on $\mathcal{A}^{(n)}$:
\begin{equation}\label{eq:vsep}
\overline{\mP}_n
:= \mathbb{E}_{\mA_n \sim p}[\mP_{\mA_n}]
= \sum_{\mA_n \in \mathcal{A}^{(n)}} p_{\mA_n} \mP_{\mA_n}.
\end{equation}
A bound on the conditional expected progress of the algorithm follows from the smallest eigenvalue of the expected projector:
\begin{equation}\label{eq:rate1}
    \expect{\norm{\bx_{k+1} - \bx_\star}^2 \mid \bx_k}
    \le
    (1-\mu_n)\norm{\bx_k-\bx_\star}^2,
    \qquad
    \mu_n := \lambda_{\min}(\overline{\mP}_n).
\end{equation}
Based on independent draws at each step, this expected progress inequality translates to geometric convergence for any initial point $\bx_0$: $\expect{\norm{\bx_{k} - \bx_\star}^2} \le (1-\mu_n)^k\norm{\bx_0-\bx_\star}^2$.

Our key result is the closed-form characterization of the expected projectors purely in terms of $\mA$ and $n$ that can be formulated recursively or in polynomial form.
\begin{theorem}[Volume Sampling Expected Projector]\label{thm:vsep}
For a fixed \(n>1\), consider the randomized alternating projections in~\EQ{eq:multi_row}, where each \(\mA_n \in \mathcal{A}^{(n)}\) is sampled with squared volume probability
\[
p_{\mA_n}
:=
\frac{v_{\mA_n}}
{\sum_{\mA_n' \in \mathcal{A}^{(n)}} v_{\mA_n'}} \qquad {\rm with} \quad v_{\mA_n} := \det(\mA_n \mA_n^T).
\]
Then the expected projector for subspaces spanned by subsets of size $n$ is given by
    \begin{equation}\label{eq:ep_thm3}
        \overline{\mP}_n := \frac{\sum_{\mA_n \in \mathcal{A}^{(n)}}{\left(v_{\mA_n}\mP_{\mA_n}\right)}}
{\sum_{\mA_n' \in \mathcal{A}^{(n)}} v_{\mA_n'}} = n\frac{\mat{\Phi}_n}{\Tr{\mat{\Phi}_n}}
    \end{equation}
where the sequence of matrices $\{\mat{\Phi}_n\}_{n=1}^N$ satisfy the recursion
    \begin{equation}\label{eq:phi_recursion}
    \mat{\Phi}_1 = {\mA^T \mA}, \qquad
    \mat{\Phi}_n = \mat{\Phi}_1 \left(\frac{\Tr{\mat{\Phi}_{n-1}}}{n-1}\I - \mat{\Phi}_{n-1}\right) \quad {\rm for} \quad n > 1.
    \end{equation}
\end{theorem}
This immediately implies Theorem~\ref{thm:volume_sampling} as the eigenvectors of $\mat{\Phi}_n$ are unchanged from those of $\mat{\Phi}_1$.

In the following subsections, we first derive a recursive expansion for a general rank-$n$ orthogonal projector, $\mP_n$, which forms the foundation of the recursion above. We then prove the theorem through a combinatorial analysis of the collection of orthogonal projectors $\mP_{\mA_n}$ associated with $\mA_n \in \mathcal{A}^{(n)}$.

\subsection{Recursive Expansion of Orthogonal Projectors}
In this subsection we discuss a recursive property of general orthogonal projectors. For that purpose, we denote a set of vectors as $\left\{\va_1, \dots, \va_n\right\}$ and a matrix whose columns are specified by those vectors as $\left[\va_1,\dots,\va_n\right]$.

\begin{lemma}\label{thm:beauty}
Let  $U$ denote the span of $n$ linearly independent vectors  $\{\va_1, \dots, \va_n\}$. For each $1 \leq s \leq n$, let $\mat{P}_1^{s} := \va_s \va_s^T / \norm{\va_s}^2$ be the orthogonal projector into the subspace spanned by $\va_s$, and $\mat{P}_{n-1}^{\bar s}$ be the orthogonal projector into the subspace $U_{\bar s}$ of $U$, spanned by all but $\va_s$. Furthermore, let the angle between $\va_s$ and $U_{\bar s}$ be denoted by $\theta_s$. Then the orthogonal projector, $\mP_n$, into $U$ has the expansion:
\begin{align}
\mat{P}_n &= \sum_{s=1}^n{\frac{1}{\sin^2{\theta_s}}  \mat{P}_1^{s} \left(\mat{I} - \mat{P}_{n - 1}^{\bar s}\right)}.
\label{eq:P-recursion}
\end{align}
\begin{proof}
Let $\mat{R}$ denote the expansion on the right hand side of \EQ{eq:P-recursion}. First, we note that for any $t \neq s$, $\mat{P}_{n-1}^{\bar s} \va_t = \va_t$, which implies that $\left(\mat{I} - \mat{P}_{n - 1}^{\bar s}\right)\va_t = \vect{0}$. We decompose $\va_s = \va_s^\parallel + \va_s^\perp$ into its components in the subspace $U_{\bar s}$ and orthogonal to that subspace: $\va_s^\parallel = \mat{P}_{n-1}^{\bar s} \va_s$ and $\va_s^\perp = \left(\mat{I} - \mat{P}_{n - 1}^{\bar s}\right) \va_s$. The length of the orthogonal component is therefore $\norm{\va_s^\perp} = \norm{\va_s} \sin \theta_s$.

We now use these observations to show that
\begin{equation*}
\mat{R} \va_t = \va_t \qquad 1 \le t \le n.
\end{equation*}
Based on the above, the only term in the summation corresponding to $\mat{R} \va_t$ that is non-zero is when $s=t$, so that
\begin{equation*}
\mat{R} \va_t= {\frac{1}{\sin^2{\theta_t}}  \mat{P}_1^{t} \va_t^\perp}
= {\frac{\inprod{\va_t}{\va_t^\perp}}{\sin^2{\theta_t}\norm{\va_t}^2} \va_t}.
\end{equation*}
Now $\inprod{\va_t}{\va_t^\perp} =\norm{\va_t} \norm{\va_t^\perp} \cos(\angle(\va_t^\perp, \va_t))$, and since $\angle(\va_t^\perp, \va_t)$ is complementary to  $\theta_t = \angle(\va_t^\parallel, \va_t)$, we get $\inprod{\va_t}{\va_t^\perp} = \norm{\va_t} \norm{\va_t^\perp} \sin \theta_t$. Plugging in $\norm{\va_t^\perp} = \norm{\va_t} \sin \theta_t$ then gives the result.

For any vector $\vect{z}$ orthogonal to $U$, we have: $\mat{P}_{n-1}^{\bar s} \vect{z} = \vect{0}$ and $\mat{P}_1^s \vect{z} = \vect{0}$ since $\vect{z}$ is orthogonal to all the $n$ vectors $ \va_1 , \ldots , \va_n$; therefore, $\mat{R} \vect{z} = \vect{0}$.

This shows $\mat{R} = \mat{P}_n$ is the orthogonal projector into $U$, which is unique.
\end{proof}
\end{lemma}

We expand on the notion of orthogonal projectors and introduce a {quasi projector} that is well-defined even when the vectors are not linearly independent:
\begin{equation}
\mat{Q}_n := \left[\va_1,\dots,\va_n\right] {\rm adj}(\mG_n) \left[\va_1,\dots,\va_n\right]^T.
\label{eq:defQ}
\end{equation}
The adjugate matrix, ${\rm adj}(\mG_n)$, is also the cofactor matrix of the Gram matrix $\mG_n := \left[\va_1,\dots,\va_n\right]^T \left[\va_1,\dots,\va_n\right]$ due to its symmetry. When the vectors are linearly dependent $\mat{Q}_n = \mat{0}$ (as we will see), otherwise the quasi projector is just a scaled version of the orthogonal projector:
\begin{align}
\mat{Q}_n = {v_n} \mat{P}_n \quad {\rm where} \quad v_n := \det \mG_n.
\end{align}
Here $v_n$ represents the {\it square} of $n$-volume of the parallelepiped formed by the column vectors of $\left[\va_1,\dots,\va_n\right]$. Using this volume definition, we can present a corollary to Lemma~\ref{thm:beauty}:
\begin{corollary}[Recursive Quasi Projector]
Under the assumptions of the lemma, let $v^s_1 := \norm{\va_s}^2$ and $v^{\bar s}_{n-1}$ denote the square volume of the parallelepiped formed by all but $\va_s$. Moreover, let $\mat{Q}^s_1 = \va_s\, \va_s^T$ and $\mat{Q}^{\bar s}_{n-1}$ denote the quasi projectors corresponding to $\mat{P}^s_1$ and $\mat{P}^{\bar s}_{n-1}$, respectively. Then we have:
\begin{equation}
\mat{Q}_n = \sum_{s=1}^n{\mat{Q}^s_1 \left( v_{n-1}^{\bar s} \mat{I} - \mat{Q}_{n-1}^{\bar s}\right)}.
\label{eq:recQ}
\end{equation}
\begin{proof}
The volume of the parallelepiped equals the product of the volume of any facet and the component of the remaining edge orthogonal to that facet. Consequently, the squared volumes satisfy $v_n = v^{\bar s}_{n-1} v^s_1 \sin^2\theta_s$.
\end{proof}
\end{corollary}

\begin{lemma}[]
\label{lemma2}
For a linearly dependent set of vectors, $\{\va_1, \dots, \va_n\}$, the quasi projector $\mat{Q}_n = \mat{0}$, the zero matrix.

\begin{proof}
Based on~\EQ{eq:defQ}, $\mat{Q}_n^T \mat{Q}_n = \mat{0}$ since $\mG_n {\rm adj}(\mG_n) = (\det \mG_n) \mat{I} = \mat{0}$. This implies that all columns of $\mat{Q}_n$ have zero norm. Hence $\mat{Q}_n = \mat{0}$. \highlight{On the other hand, since Corollary 1 is under the linearly independent assumption, we show, using a continuity argument, that for a linearly dependent set Corollary 1 still holds.

The quasi projector, defined by~\EQ{eq:defQ} $\mat{Q}_n = \left[\va_1,\dots,\va_n\right] {\rm adj}(\mG_n) \left[\va_1,\dots,\va_n\right]^T$, has elements that are each a continuous function of vectors $\va_1, \dots, \va_n$. This follows from the elements of the adjugate matrix (minors) being signed volumes of subsets of $\va$'s. Volume (determinant) is a continuous function of its set of vectors. Corollary 1 can therefore be expanded to include linearly dependent sets since any linearly dependent set of $n \le N$ vectors in $\mathbb R^N$ can be perturbed to become a linearly independent set.}
\end{proof}
\end{lemma}

\subsection{Combinatorial Analysis}
We now return to the expectation in Theorem~\ref{thm:vsep}, taken over
all \(\mA_n\in\mathcal{A}^{(n)}\). Recall that
\(\mP_{\mA_n}\) and \(v_{\mA_n}\) denote, respectively, the orthogonal
projector and squared volume associated with \(\mA_n\), and that the
corresponding quasi-projector satisfies $\mQ_{\mA_n}=v_{\mA_n}\mP_{\mA_n}$. Thus, the expected projector in~\EQ{eq:ep_thm3} can be written as
\[
\overline{\mP}_n
=
\frac{
\sum_{\mA_n\in\mathcal{A}^{(n)}}\mQ_{\mA_n}
}{
\sum_{\mA_n\in\mathcal{A}^{(n)}}v_{\mA_n}
}.
\]
We define for each $1 \le n \le N$ the sum of quasi projectors as
\begin{align}\label{eq:phi_def}
\mat{\Phi}_n
:=
\sum_{\mA_n\in\mathcal{A}^{(n)}}\mQ_{\mA_n}
\end{align}
and, noting that $\GG_1 = \sum_{\mA_1 \in \mathcal{A}^{(1)}}\mA_1^T\mA_1 = \sum_{\va^T \in \mathcal{A}}{\va \va^T} = \mA^T \mA$, show that the sequence
\(\{\mat{\Phi}_n\}_{n=1}^N\) satisfies the recursion
in~\EQ{eq:phi_recursion}. First we show that the normalization constant in the numerator of $\overline{\mP}_n$ is simply calculated from the trace of $\GG_n$.

\begin{proposition}[Normalization constant for volume probabilities]\label{prop:trace}
    The sum of squared volumes over subsets of size $n$ can be computed from $\mat{\Phi}_n$ by
    \begin{equation*}
        \sum_{\mA_n \in \mathcal{A}^{(n)}}{v_{\mA_n}} = \frac{\Tr{\mat{\Phi}_n}}{n}.
    \end{equation*}
\end{proposition}
\begin{proof}
    Since $\Tr{\mQ_{\mA_n}} = v_{\mA_n} \Tr{\mP_{\mA_n}} = v_{\mA_n} {\rm rank}(\mA_n)$ the result follows since $v_{\mA_n}$ is non-zero only if ${\rm rank}(\mA_n) = n$.
\end{proof}

This allows us to complete the proof of our key theorem on volume sampling.
\begin{proof}[Proof of Theorem~\ref{thm:vsep}]
Using the normalization constant provided in Proposition~\ref{prop:trace}, we have that $\overline{\mP}_n = n \frac{\GG_n}{\Tr \GG_n}$ establishing~\EQ{eq:ep_thm3}. To establish the theorem, we show that the sum of quasi projectors defined in~\EQ{eq:phi_def} satisfy the recursion~\EQ{eq:phi_recursion} for $n > 1$ as:
\begin{align*}
\GG_{n} = \GG_1 \left(\frac{\Tr{\GG_{n-1}}}{n-1} \mat{I} - \GG_{n-1}\right).
\end{align*}

Recall that we denote the set $\mathcal A$ containing rows $\va^T \in \mathcal A$ of the matrix $\mA$. Before summing
the quasi projectors as in~\EQ{eq:phi_def} over the $\binom{M}{n}$ submatrices in
$\mathcal A^{(n)}$, we first show a similar recursion for summation over the Cartesian product $\mathcal A^n = \mathcal A\times\cdots\times\mathcal A$, consisting of all $M^n$ ordered $n$-tuples of rows of $\mA$, with
repetitions allowed. Then we argue using symmetries the recursion in this proposition. Accordingly, we define
\begin{equation*}
\widetilde{\GG}_n := \sum_{\mA_n \in \mathcal{A}^n}\mQ_{\mA_n} = \sum_{\va_1^T \in \mathcal{A}} \sum_{\va_2^T \in \mathcal{A}} \cdots \sum_{\va_n^T \in \mathcal{A}} \mat{Q}_n,
\end{equation*}
where $\mQ_n$ denotes the quasi projector corresponding to the set $\{\va_1,\cdots,\va_n\}$ as defined in~\EQ{eq:defQ}. In this expansion to $\mathcal{A}^n$ the relationship between trace and sum of squared volumes, established in Proposition~\ref{prop:trace} for $\mathcal{A}^{(n)}$, still holds: $\Tr{\widetilde{\GG}}_n = {n}\sum_{\mA_n \in \mathcal{A}^n}{v_{\mA_n}}$. The expanded choices allow for summation over individual vectors that can sift through the recursion in~\EQ{eq:recQ}. For example, the first term in~\EQ{eq:recQ} for $s=1$ shows:
\begin{align*}
&\sum_{\va_1^T \in \mathcal{A}} \sum_{\va_2^T \in \mathcal{A}} \cdots \sum_{\va_n^T \in \mathcal{A}} \mat{Q}_1^1 \left( v^{\bar 1}_{n-1} \mat{I} - \mat{Q}^{\bar 1}_{n-1}\right) \\
&= \sum_{\va_1^T \in \mathcal{A}} \mat{Q}^1_1 \sum_{\va_2^T, \dots, \va_n^T \in \mathcal{A}} \left( v^{\bar 1}_{n-1} \mat{I} -  \mat{Q}^{\bar 1}_{n-1}\right) \\
&= \sum_{\va^T \in \mathcal{A}} \mat{Q}_1 \sum_{\va_1^T, \dots, \va_{n-1}^T \in \mathcal{A}} \left( v_{n-1} \mat{I} -  \mat{Q}_{n-1}\right) \\
&= \sum_{\mA_1 \in \mathcal{A}^{(1)}} \mat{Q}_{\mA_1} \sum_{\mA_{n-1} \in \mathcal{A}^{n-1}}{\left(v_{\mA_{n-1}}\I - \mQ_{\mA_{n-1}} \right)} \\
&= \GG_1 \left( \frac{\Tr{\widetilde{\GG}_{n-1}}}{n-1} \mat{I} - \widetilde{\GG}_{n-1} \right).
\end{align*}
Observing that the result of this summation is independent of $s$, allows us to establish:
\begin{equation*}
\widetilde{\GG}_n = n \GG_1 \left(\frac{\Tr{\widetilde{\GG}_{n-1}}}{n-1} \mat{I} - \widetilde{\GG}_{n-1}\right).
\end{equation*}
Now we observe that in a particular choice $\mA_n$ of $n$-rows with replacement, if any row of $\mA$ is selected more than once, then $v_{\mA_n} = 0$ and $\mQ_{\mA_n} = \mat{0}$ for any $n$. This means the summation can be restricted to the set $\mathcal A^{[n]} := \left\{[\va_{i_1}\ \cdots\ \va_{i_n}]^T : 1\le i_1,\ldots,i_n\le M,\; i_j\neq i_k\ \text{for } j\neq k \right\}$, consisting of the $M!/(M-n)!$ ordered selections of $n$ distinct rows of $\mA$, while leaving $\widetilde{\GG}_n$ unchanged:
\begin{equation*}
\sum_{\mA_n \in \mathcal{A}^{[n]}}\mQ_{\mA_n} = \widetilde{\GG}_n = n \GG_1 \left(\frac{\Tr{\widetilde{\GG}_{n-1}}}{n-1} \mat{I} - \widetilde{\GG}_{n-1}\right).
\end{equation*}
To further shrink the space of choices to $n$-combinations, $\mathcal{A}^{(n)}$, we note that permuting the order of the rows in a particular choice does not change the {squared volume}, $v_{\mA_n}$ of the parallelepiped they form. Moreover, permuting the rows in a particular choice does not change the orthogonal projector $\mat{P}_{\mA_n}$ and consequently $\mat{Q}_{\mA_n}$. This means for any $n \ge 1$, we have $\GG_n = \widetilde{\GG}_n/n!$  and consequently:
\begin{align*}
\GG_n &= \frac{1}{n!}\widetilde{\GG}_n = \frac{n}{n!} \GG_1 \left((n-1)!\, \frac{\Tr{{\GG}_{n-1}}}{n-1} \mat{I} - (n-1)! \GG_{n-1}\right)\\
& = \GG_1 \left(\frac{\Tr{{\GG}_{n-1}}}{n-1}\mat{I} - \GG_{n-1}\right).
\end{align*}
\end{proof}

This sequence of matrices $\{\GG_n\}_{n=1}^N$, originating with $\GG_1 = \mathbf{A}^T \mathbf{A}$ and terminating at $\GG_N = \mathbf{I}$, gives the recursive form of the Cayley-Hamilton theorem used to compute the coefficients of the characteristic polynomial of $\mathbf{A}^T\mathbf{A}$ in the Faddeev-LeVerrier algorithm. Unwinding the recursion, we obtain the lowest $n$ terms of the characteristic polynomial of $\GG_1 = \mA^T \mA$ in terms of its powers and traces:
\begin{corollary}[Polynomial Form]
\begin{align*}\label{eq:ggPoly}
\GG_n = \sum_{p=1}^n{(-1)^{p-1} \frac{\Tr{{\GG}_{n-p}}}{n-p} \GG_1^p }.
\end{align*}
\end{corollary}

We can now extend the results of Theorem~\ref{thm:volume_sampling} to establish our relaxed uniform sampling results.
\begin{proof}[Proof of Theorem~\ref{thm:uniform_sampling}]

To establish Theorem~\ref{thm:uniform_sampling}, we consider the centered form of~\EQ{eq:relaxed}:
\begin{equation*}
    \bx_{k+1} - \bx_\star = (\I - \omega_{\mA_n}\mP_{\mA_n}) (\bx_{k} - \bx_\star) \qquad {\rm with} \quad \omega_{\mA_n} := 1 + \sqrt{1 - \frac{v_{\mA_n}}{v_n^{\max}}}.
\end{equation*}
For each choice of $\mA_n \in \mathcal{A}^{(n)}$ for the $k^{\rm th}$-step, the progress can be measured by 
\begin{align*}
    \norm{\bx_{k+1} - \bx_\star}^2 &= (\bx_{k} - \bx_\star)^T\left(\I - \omega_{\mA_n}(2 - \omega_{\mA_n})\mP_{\mA_n}\right) (\bx_{k} - \bx_\star)\\
    & = (\bx_{k} - \bx_\star)^T\left(\I - \frac{v_{\mA_n}}{v_n^{\max}}\mP_{\mA_n}\right) (\bx_{k} - \bx_\star)\\
    &= (\bx_{k} - \bx_\star)^T\left(\I - \frac{1}{v_n^{\max}}\mQ_{\mA_n}\right) (\bx_{k} - \bx_\star).
\end{align*}
The conditional expected progress of the relaxed uniform iterations, in each step, for the uniform sampling distribution $p_{\mA_n} := 1 / |\mathcal{A}^{(n)}|$ for choosing $\mA_n \in \mathcal{A}^{(n)}$, is characterized by 
\begin{equation*}
    \mathbb{E}_{\mA_n \sim u}[{\norm{\bx_{k+1} - \bx_\star}^2 \mid \bx_k}] = (\bx_k - \bx_\star)^T (\I - \frac{1}{v_n^{\max}}\overline{\mQ}_n) (\bx_k - \bx_\star), 
\end{equation*}
with the expected quasi projector defined as \begin{equation*}
    \overline{\mQ}_n := \frac{\sum_{\mA_n \in \mathcal{A}^{(n)}}{\mQ_{\mA_n}}}{|\mathcal{A}^{(n)}|} = \frac{\GG_n}{|\mathcal{A}^{(n)}|}.
\end{equation*}
According to~\EQ{eq:ep_thm3} in Theorem~\ref{thm:volume_sampling} we have $\GG_n = \overline{\mP}_n \sum_{\mA_n \in \mathcal{A}^{(n)}}{v_{\mA_n}}$ which implies $\overline{\mQ}_n = \bar{v}_n \overline{\mP}_n$ with the mean square volume $\bar{v}_n = \sum_{\mA_n \in \mathcal{A}^{(n)}}{v_{\mA_n}} / |\mathcal{A}^{(n)}|$ as defined in Theorem~\ref{thm:uniform_sampling}. This provides
\begin{equation*}
    \widetilde{\mu}_n = \frac{1}{v_n^{\max}}\lambda_{\min}(\overline{\mQ}_n) = \frac{\bar{v}_n}{v_n^{\max}}\lambda_{\min}(\overline{\mP}_n) = \frac{\bar{v}_n}{v_n^{\max}}\mu_n.
\end{equation*}
\end{proof}

\section{Numerical Experiments}\label{sec:experiments}

The numerical validations are organized into two parts. The first compares volume sampling, ordinary uniform sampling, and relaxed uniform sampling on problems for which exact volume-sampling probabilities can be computed. These experiments provide empirical validation of the convergence behavior predicted by Theorems~\ref{thm:volume_sampling} and~\ref{thm:uniform_sampling}. The second part considers larger-scale problems, where exact volume sampling is computationally impractical, and compares only uniform sampling and relaxed uniform sampling to assess the practical scalability of the proposed approach.

\subsection{Small-Scale Comparison}
\label{sec:small_scale_comparison}

The purpose of this experiment is to compare volume sampling, uniform sampling, and relaxed uniform sampling in a regime where the exact volume-sampling probabilities can be computed. This permits a direct comparison between the observed convergence behavior and the theoretical convergence parameters $\mu_n$ and $\widetilde{\mu}_n$.

We consider four row-normalized test matrices. The first is a Gaussian matrix $\mA\in\mathbb{R}^{20\times10}$ whose entries are drawn independently from a standard normal distribution. The remaining three are the $20\times20$ Parter, Prolate (with parameter $w=0.45$), and Wilkinson matrices from the classical matrix-gallery literature. After generation, each nonzero row is normalized to have unit Euclidean norm.

For each matrix, we generate a consistent linear system $\mA\bx_\star=\vect b$, where $\bx_\star$ is drawn from a standard Gaussian distribution, and initialize the iteration with $\bx_0=\vect0$. We consider subset sizes $n\in\{1,2,4\}$. Volume sampling selects subsets with probability proportional to their squared volume, whereas uniform and relaxed uniform sampling use the uniform distribution over all subsets of size $n$. Relaxed uniform sampling employs the subset-dependent relaxation parameter of Theorem~\ref{thm:uniform_sampling}. To reduce  variability, uniform and relaxed uniform sampling use identical sampled subset sequences in each trial.

We report the mean relative squared error $\mathbb{E}\|\bx_k-\bx_\star\|^2/\|\bx_0-\bx_\star\|^2$, averaged over $250$ independent trajectories. The Gaussian experiment is run for $250$ iterations, the Parter and Wilkinson experiments for $500$ iterations, and the Prolate experiment for $4000$ iterations to accommodate its slower convergence. For each matrix and subset size, we also compute the theoretical convergence parameters $\mu_n$ and $\widetilde{\mu}_n=(\bar v_n/v_n^{\max})\mu_n$.

\begin{figure}[t]
    \centering
    \includegraphics[width=\linewidth]{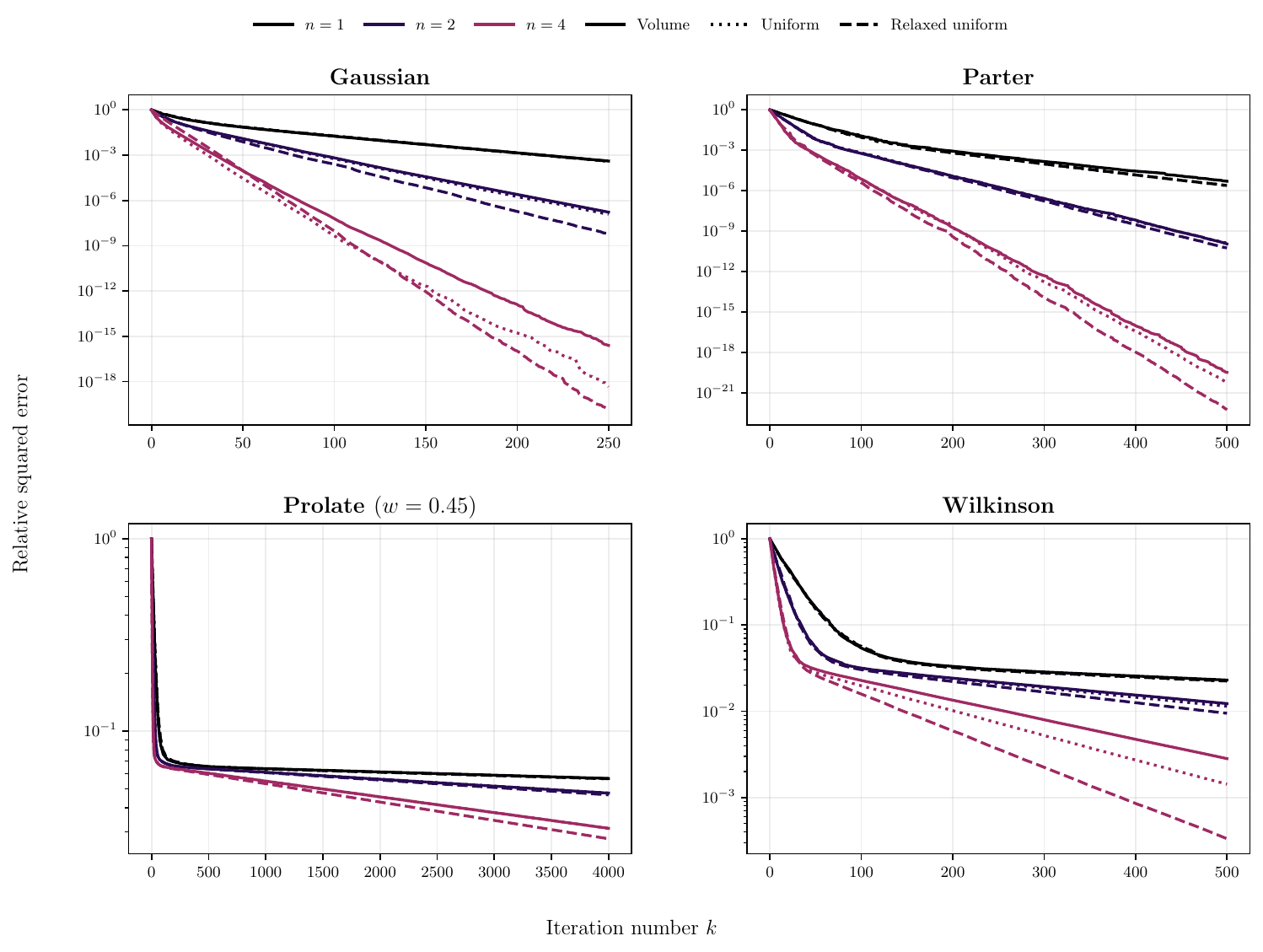}
    \caption{Small-scale comparison of volume sampling, uniform sampling, and relaxed uniform sampling on row-normalized Gaussian, Parter, Prolate, and Wilkinson matrices. Colors distinguish the subset sizes $n=1,2,4$. Solid, dotted, and dashed curves correspond to volume sampling, uniform sampling, and relaxed uniform sampling, respectively. Each curve shows the mean relative squared error on a logarithmic scale, averaged over $250$ independent trajectories.}
    \label{fig:small_scale_comparison}
\end{figure}

Figure~\ref{fig:small_scale_comparison} compares the empirical convergence of the three sampling strategies. The corresponding theoretical convergence parameters are reported in Table~\ref{tab:small_scale_rate_parameters}.

\begin{table}[h]
    \centering
    \caption{Theoretical convergence parameters for the matrices used in the small-scale experiments. Here, $\mu_n$ is the convergence parameter for volume sampling and $\widetilde{\mu}_n=(\bar v_n/v_n^{\max})\mu_n$ is the corresponding parameter for relaxed uniform sampling.}
    \label{tab:small_scale_rate_parameters}
    \begin{tabular}{lccc}
        \toprule
        Matrix & $n$ & $\mu_n$ & $\widetilde{\mu}_n$ \\
        \midrule
        \begin{filecontents*}{figure_2_rate_parameters.csv}
matrix_name,subset_size,mu,mean_to_max_volume_ratio,tilde_mu
Gaussian,1,0.010987514795774669,0.9999999999999998,0.010987514795774667
Gaussian,2,0.025470307411339994,0.8982017122777234,0.02287747372910557
Gaussian,4,0.07164043397976165,0.51571500219917,0.03694604656742227
Parter,1,0.009274596796907005,0.9999999999999998,0.009274596796907003
Parter,2,0.019387653607089238,0.9981970285779369,0.019352698221694796
Parter,4,0.04262193106552776,0.9886852628200536,0.04213967511741952
Prolate $(w=0.45)$,1,1.9583417318269282e-05,0.9999999999999997,1.9583417318269275e-05
Prolate $(w=0.45)$,2,4.1476801793932836e-05,0.9939879117050459,4.12274395993554e-05
Prolate $(w=0.45)$,4,9.402593302656013e-05,0.9714121317904227,9.13379320449143e-05
Wilkinson,1,0.0005199568728563832,0.9999999999999998,0.0005199568728563831
Wilkinson,2,0.0011186809008155848,0.9780062201696567,0.0010940768793826366
Wilkinson,4,0.002615206737296353,0.8732150968509611,0.002283638004393521
\end{filecontents*}
\csvreader[
            late after line=\\
        ]{figure_2_rate_parameters.csv}{}{%
            \csvcoli
            & \csvcolii
            & \num{\csvcoliii}
            & \num{\csvcolv}
        }
        \bottomrule
    \end{tabular}
\end{table}

Table~\ref{tab:small_scale_rate_parameters} reports the theoretical convergence parameters corresponding to the two convergence theorems. For the matrices considered here, both \(\mu_n\) and \(\widetilde{\mu}_n\) increase with the subset size, although for \(\widetilde{\mu}_n\) this monotonic behavior is not guaranteed in general. As predicted by Theorem~\ref{thm:uniform_sampling}, \(\widetilde{\mu}_n\le\mu_n\), since \(\bar v_n/v_n^{\max}\le1\). Nevertheless, these quantities are convergence bounds rather than exact convergence rates, and, in all of our experiments, the relaxed uniform sampling scheme converges faster than volume sampling despite its weaker theoretical guarantee. These values therefore serve as a theoretical baseline for interpreting the empirical convergence observed in Figure~\ref{fig:small_scale_comparison}.

\subsection{Large-Scale Experiments}

The purpose of this experiment is to investigate the practical scalability of relaxed uniform sampling. For large matrices, exact volume sampling becomes computationally infeasible because of the combinatorial growth in the number of subsets. We therefore compare only uniform sampling and relaxed uniform sampling to determine whether the improvements observed in the small-scale experiments persist at practically relevant problem sizes.

We consider two sparse overdetermined systems. The first, \texttt{fanbeam\_64}, is a fan-beam CT forward model for reconstructing a $64\times64$ image from $64$ projection views and $96$ detector elements per view. The resulting matrix has dimensions $5448\times4096$, with $341\,472$ nonzero entries and density $1.53\times10^{-2}$. The second, \texttt{well1850}, is from the SuiteSparse Matrix Collection and has dimensions $1850\times712$, with $8755$ nonzero entries and density $6.65\times10^{-3}$.

For each matrix, we consider subset sizes $n\in\{10,20,40\}$. Uniform and relaxed uniform sampling use identical randomly generated subset sequences and are averaged over ten independent trials. Each trial is run for $500$ equivalent row passes. The exact solution and initial iterate are identical for both methods so that the comparison isolates the effect of the adaptive relaxation.

\begin{figure}[t]
    \centering
    \includegraphics[width=\textwidth]{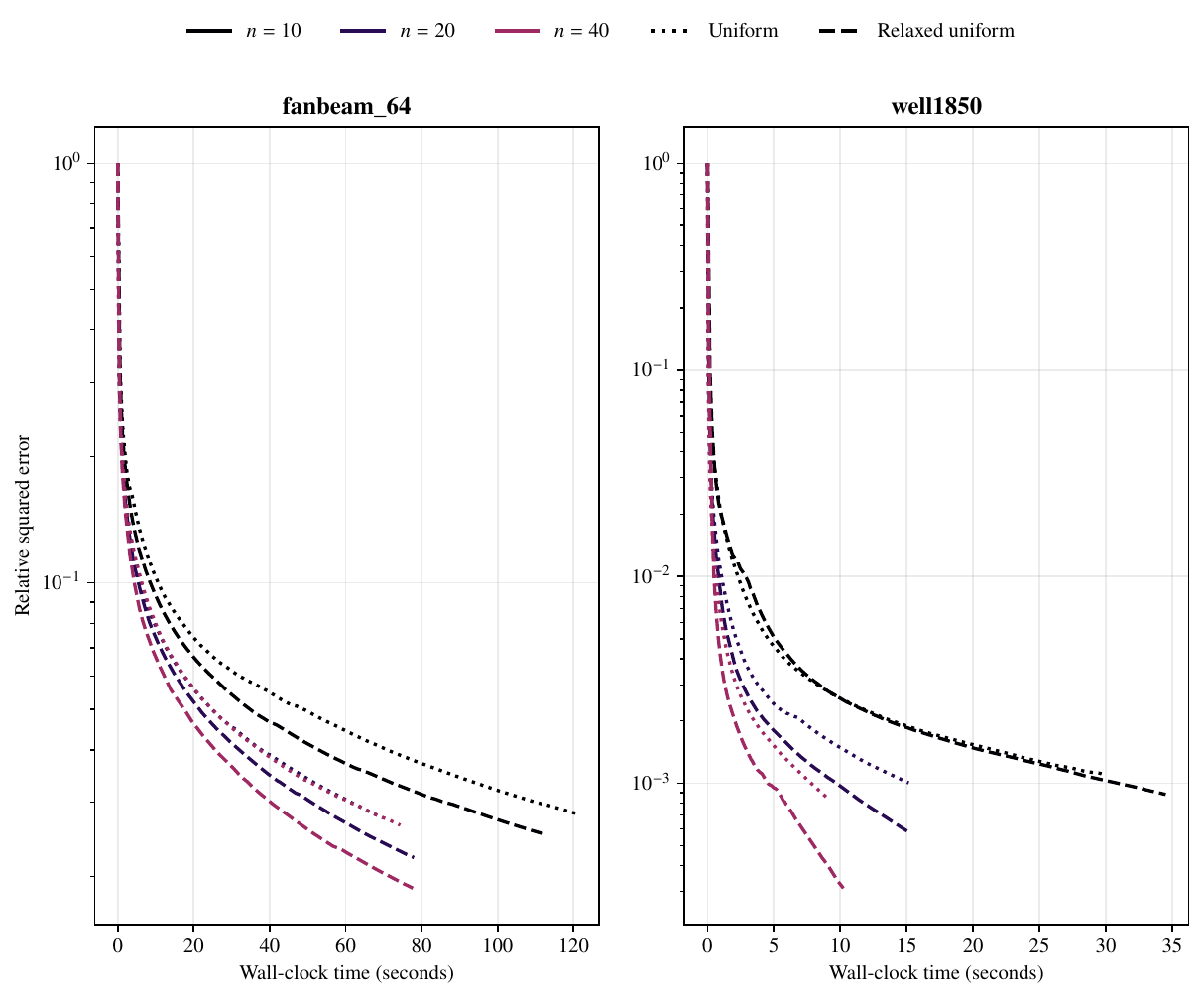}
    \caption{Relative squared error versus wall-clock time for \texttt{fanbeam\_64} and \texttt{well1850}. Colors distinguish the block sizes $n\in\{10,20,40\}$, while dotted and dashed curves correspond to uniform and relaxed uniform sampling, respectively. Each curve is averaged over ten independent trials.}
    \label{fig:large-scale-wall-clock}
\end{figure}

Figure~\ref{fig:large-scale-wall-clock} shows the relative squared error as a function of wall-clock time. For both matrices, relaxed uniform sampling generally reaches a lower error than standard uniform sampling in the same computational time, with the advantage becoming more pronounced for larger subset sizes. These results demonstrate that the reduction in iteration count obtained by the adaptive relaxation is preserved in terms of wall-clock performance. In our implementation, the squared subset volume, $v_{\mA_n}$, is obtained directly from the Cholesky factorization used to compute the projection $\mP_{\mA_n}$, avoiding a separate determinant computation.

\section{Related Work and Discussion}\label{sec:discussion}

Volume sampling has become an important tool in randomized numerical linear algebra and theoretical computer science. Originally introduced for low-rank matrix approximation and projective clustering~\citep{deshpande2006matrix}, it has subsequently been applied to linear regression, where it yields unbiased least-squares estimators~\citep{derezinski2017unbiased}, experimental design~\citep{nikolov2022proportional}, and randomized optimization~\citep{rodomanov-kropotov}. A substantial body of work has also been devoted to the development of efficient exact and approximate sampling algorithms, including reverse iterative and leveraged volume sampling~\citep{derezinski2018reverse}, reflecting both the theoretical importance and the computational challenges of determinantal sampling. Our work contributes to this literature by providing an exact spectral characterization of the expected projector associated with volume sampling and using this characterization to expose structural properties of volume-sampled iterations that are not visible from existing convergence bounds.

The work most closely related to ours is that of Rodomanov and Kropotov~\citep{rodomanov-kropotov}, who introduced volume sampling into randomized block coordinate descent. Their convergence guarantees are governed by the expectation of an embedded inverse principal submatrix associated with the sampled coordinates, which plays the same role as the expected projector in our analysis. In Lemmas~3.1--3.3 of their work, they derived an exact spectral representation for this expectation, in which the spectral coefficients are given by ratios of elementary symmetric polynomials of the eigenvalues. Although this representation eliminates the expectation itself, its spectral coefficients remain combinatorial objects, making the resulting expression difficult to work with in the convergence analysis. They therefore derived a tractable Loewner-order lower bound for the expectation, which serves as the foundation for their convergence analysis. Similar Loewner-order bounds for expected inverse principal submatrices have also played a central role in the analysis of stochastic dual Newton methods~\citep{qu-richtarik-takac-fercoq} and randomized block Gauss--Seidel methods~\citep{tu2017breaking}.

Derezi\'nski and Yang~\cite{derezinski-yang} subsequently sharpened the convergence analysis based on the Rodomanov--Kropotov representation. For a prescribed block size $k'$, their Lemma~4.1 provides a lower bound on the smallest positive eigenvalue of the expected projector by introducing an auxiliary index $k<k'$ and a corresponding tail condition measure $\bar{\kappa}_k$. This leads to their top-$k$-deflation interpretation, with particularly strong consequences when a relatively small number of leading singular values is followed by a favorable, nearly flat tail. Thus, their result provides a sufficient convergence guarantee expressed through selected spectral information, rather than a characterization of how volume sampling transforms the spectrum itself.

In contrast, Theorem~\ref{thm:vsep} shows that, for volume sampling, the Loewner-order relaxation is unnecessary. The expected projector admits an exact characterization through a simple matrix-polynomial recursion rooted in the Cayley--Hamilton theorem, revealing an unexpected connection between volume sampling and classical matrix theory while directly determining both the expectation and its complete spectrum. In particular, if $\mu_n$ denotes the smallest eigenvalue of the expected projector for blocks of size $n$, then the standard per-iteration expected convergence factor is $1-\mu_n$. Our characterization determines $\mu_n$ exactly from the spectrum of the underlying problem and, more importantly, shows that $\mu_n$ increases monotonically with $n$. Consequently, every increase in the volume-sampled block size provably improves the per-iteration convergence rate.

This monotonicity is a structural property of volume sampling rather than a generic consequence of batching. In randomized mini-batch methods, for example, the benefit of increasing the batch size can depend on the problem and can saturate beyond a problem-dependent critical batch size~\citep{jain-kakade-kidambi-netrapalli-sidford}. Volume sampling therefore exhibits a particularly strong relationship between computational effort per iteration and convergence: increasing the block size always improves the per-iteration contraction guaranteed by the exact expected projector. This property follows directly from the spectral transformation characterized in Theorem~\ref{thm:vsep} and provides a qualitative consequence of spectral equalization that cannot be inferred merely from a lower bound on $\mu_n$.

The complete spectral characterization is also important beyond the standard per-iteration bound. The expected-progress inequality in Theorem~1, like analogous one-step bounds for randomized iterative methods, is well known to provide a conservative estimate of observed convergence behavior~\citep{agaskar2014randomized,bai-wang, tu-recht}. The asymptotic analysis~\cite{entezaribanerjee2026} addresses this gap by leveraging the {\em two smallest eigenvalues} of the expected projector to obtain a sharper bound on the long-run convergence rate. Such an analysis requires spectral information beyond $\mu_n$ alone. The exact characterization of the expected projector in Theorem~\ref{thm:vsep} therefore enables an asymptotic convergence analysis that is inaccessible from existing lower bounds on its smallest eigenvalue alone.

Since randomized block coordinate descent for quadratic objectives is equivalent to randomized block Gauss--Seidel under the standard primal--dual correspondence, our characterization immediately strengthens the analysis of volume-sampled coordinate descent methods. More broadly, our result provides an exact first-moment characterization that can be incorporated into analyses of accelerated randomized block methods whose convergence depends on the expected projector operator~\citep{derezinski-lejeune-needell-rebrova}. Viewed in this way, the spectral equalization phenomenon is not simply an interpretation of an improved convergence bound: it describes how the entire spectrum of the expected projector evolves with block size and explains the progressive improvement in conditioning induced by volume sampling.

Block variants of classical iterative methods have a long history in numerical linear algebra, including block Gauss--Seidel, block Jacobi, and block Kaczmarz methods~\citep{saad2003iterative,elfving1980block}. Their randomized counterparts have subsequently received considerable attention~\citep{needell2014paved,liu2016accelerated,gower-richtarik,tu2017breaking}. A central challenge in the analysis of block methods is that their convergence rates typically depend on the conditioning of the individual blocks induced by a prescribed partition of $\mA$. Consequently, the theoretical guarantees are often governed by the worst-conditioned block, a quantity that may be only weakly related to the spectrum of $\mA$ itself and may require carefully constructed row pavings or coordinate partitions. Volume sampling fundamentally changes this perspective by selecting blocks according to their geometric diversity rather than a fixed partition. As a result, the expected convergence rate depends only on global spectral properties of $\mA$, eliminating the explicit dependence on the conditioning of the worst block.

Despite its attractive theoretical properties, the computational cost of volume sampling remains a major obstacle to its use in large-scale iterative methods. A substantial body of work has therefore focused on the development of efficient exact and approximate sampling algorithms. Deshpande and Rademacher~\citep{deshpande2010efficient} proposed polynomial-time algorithms for exact and approximate volume sampling based on standard linear-algebraic techniques, while subsequent work developed Markov-chain methods for related determinantal distributions and established polynomial mixing-time guarantees~\citep{li2017polynomial}. Nevertheless, the computational cost of determinant evaluations, repeated matrix factorizations, and preprocessing remains prohibitive for many large-scale numerical linear algebra problems, where the cost of a single iteration is expected to be nearly linear in the number of nonzeros.

This computational challenge has motivated specialized algorithms for small subset sizes. In particular, Rodomanov and Kropotov~\citep{rodomanov-kropotov} developed an efficient implementation for the practically important case of two-dimensional blocks, and similar ideas have recently been employed in accelerated volume-sampled randomized block Kaczmarz methods~\citep{xiang2025}. Our relaxed uniform sampling strategy takes a different approach. Rather than sampling from the determinantal distribution exactly, it combines inexpensive uniform subset selection with a relaxation that preserves much of the spectral benefit of volume sampling while avoiding its computational overhead. As demonstrated by our experiments, relaxed uniform sampling consistently outperforms existing methods in wall-clock time and remains practical on large-scale problems, while its theoretical contraction guarantee remains close to that of exact volume sampling.

These results suggest that the principal benefit of volume sampling is not the sampling distribution itself, but the spectral equalization mechanism that it induces, which can be approximated efficiently without exact determinantal sampling.
\printbibliography

\end{document}